\documentclass[12pt]{amsart}
\usepackage{color}
\usepackage{bbm}

\usepackage{amssymb}
\usepackage{graphicx}  
\DeclareGraphicsExtensions{.pstex,.eps}

\newcommand{\R}{{\mathbb{R}}}
\newcommand{\M}{{\mathcal{M}}}

\newcommand{\CC}{{\mathbb C}}

\newcommand{\Q}{{\mathcal Q}}
\newcommand{\RR}{{\mathbb R}}

\newcommand{\p}{\partial} 
\newcommand{\supp}{\operatorname{supp}}

\renewcommand{\Im}{\mathop{\rm Im}\nolimits}

\newcommand{\lX}{l^1 X}
\newcommand{\lY}{l^1 Y}

\theoremstyle{plain}

\newtheorem{thm}{Theorem}
\newtheorem{prop}{Proposition}[section]

\newtheorem{lemma}[prop]{Lemma}

\theoremstyle{definition}

\numberwithin{equation}{section}

\def\squarebox#1{\hbox to #1{\hfill\vbox to #1{\vfill}}}

\newcommand{\la}{\langle}
\newcommand{\ra}{\rangle}

\newcommand{\uz}{u_0}
\usepackage{amsxtra}

\title[Quasilinear Schr\"odinger equations]
{Quasilinear Schr\"odinger equations I:  Small data and quadratic
interactions}

\author[J.L. Marzuola]
{Jeremy L. Marzuola}

\author[J. Metcalfe]
{Jason Metcalfe}

\author[D. Tataru]
{Daniel Tataru}

\address{Department of Mathematics, University of North Carolina-Chapel Hill \\
Phillips Hall, Chapel Hill, NC  27599, USA}
\email{marzuola@email.unc.edu}

\address{Department of Mathematics, University of North Carolina-Chapel Hill \\
Phillips Hall, Chapel Hill, NC  27599, USA}
\email{metcalfe@email.unc.edu}

\address{Mathematics Department, University of California \\
Evans Hall, Berkeley, CA 94720, USA}
\email{tataru@math.berkeley.edu}

\begin{document}

\begin{abstract}
  In this article we prove local well-posedness in low-regularity
  Sobolev spaces for general quasilinear Schr\"odinger
  equations. These results represent improvements in the small data regime of the pioneering
  works by Kenig-Ponce-Vega and Kenig-Ponce-Rolvung-Vega,
where viscosity methods were used to prove existence
of solutions in very high regularity spaces.  Our arguments here are
purely dispersive.  The function spaces in which we show
existence are constructed in ways motivated by the results of
Mizohata, Ichinose, Doi, and others, including
the authors.
\end{abstract}

\maketitle 

\section{Introduction}


In this article we consider the local well-posedness for quasilinear
Schr\"odinger equations
\begin{equation}
\label{eqn:quasiquad}
\left\{ \begin{array}{l}
i u_t + g^{jk} (u,\nabla u ) \p_j \p_ku = 
F(u,\nabla u) , \quad u:
\RR \times \RR^d \to \CC^m \\ \\
u(0,x) = u_0 (x)
\end{array} \right. 
\end{equation}
with small initial data in a space with relatively low Sobolev
regularity but with some extra decay assumptions. Here  
\[
g : \CC^m \times (\CC^m)^d \to
\RR^{d \times d}, \qquad 
F: \CC^m \times (\CC^m)^d \to \CC^m
\]
are smooth functions which satisfy
\begin{equation}\label{gF}
g(0) = I_d, \qquad 
F(y,z)= O(|y|^2+|z|^2) \text{ near } (y,z) = (0,0).  
\end{equation}

We also consider a second class of  quasilinear
Schr\"odinger equations 
\begin{equation}
\label{eqn:quasiquad1}
\left\{ \begin{array}{l}
i u_t + \p_j g^{jk} (u)  \p_ku = 
F(u,\nabla u) , \ u:
\RR \times \RR^d \to \CC^m \\ \\
u(0,x) = u_0 (x),
\end{array} \right. 
\end{equation}
with $g$ and $F$ as in \eqref{gF} but where the metric $g$ depends on
$u$ but not on $\nabla u$.  Such an equation is obtained for instance
by differentiating the first equation
\eqref{eqn:quasiquad}. Precisely, if $u$ solves \eqref{eqn:quasiquad}
then the vector $(u,\nabla u)$ solves an equation of the form
\eqref{eqn:quasiquad1}, with a nonlinearity $F$ which depends 
at most quadratically on $\nabla u$.


We remark that the second order operator in \eqref{eqn:quasiquad1} is
written in divergence form. This is easily achieved by commuting the
first derivative with $g$ and moving the outcome to the right hand
side. However, the second order operator in \eqref{eqn:quasiquad}
cannot be written in divergence form without changing the type
of the equation.

Naively one might at first consider the well-posedness of these
problems in Sobolev spaces $H^s(\R^d)$ with large enough $s$.  This is,
for instance, what is done in the case of quasilinear wave equations,
using energy estimates, Sobolev embeddings and Gr\"onwall's inequality
as in \cite{Hor,Sog}.  However, this cannot work in general for the
above Schr\"odinger equations.  

The obstruction comes from the infinite speed of propagation phenomena.
From \cite{Miz1,Miz2,Miz3,Ich,MMT1}, it is
known that even in the case of linear problems of the form
\begin{equation} \label{lin}
(i \partial_t  +  \Delta_g)  v = A_i(x) \partial_i v,
\end{equation}
a necessary condition for $L^2$ well-posedness is an integrability
condition for the magnetic potential $A$ along the Hamilton flow of
the leading order differential operator.  In the case of
\eqref{eqn:quasiquad}, we would have to look instead at the
corresponding linearized problem, which would exhibit a magnetic
potential of the form $A = A(u,\nabla u)$.  Such a potential in
general does not satisfy Mizohata's integrability condition even if
$A(u) = u$ or $A(u) = \nabla u$ with $u$ solving the linear constant
coefficient Schr\"odinger equation with $H^s$ initial data and $s$
arbitrarily large.


Given the above considerations, it is natural to add some decay to the
$H^s$ Sobolev spaces where the quasilinear problem
\eqref{eqn:quasiquad} is considered. A traditional way to do that is
to use weighted $H^s$ spaces with polynomial weights. This avenue was
pursued for instance in \cite{KPV,KPRV1,KPRV2}, where the first local
well-posedness results for this problem were obtained for solutions in
$H^s \cap L^2 (\langle x \rangle^N)$, where $\langle x
\rangle=(1+|x|^{2})^{\frac{1}{2}}$, for some unspecified sufficiently
large $s$ and $N$ depending upon complicated
asymptotics.


One disadvantage of the above approach is that the results are not
invariant with respect to translations.  In this article we propose a
different set-up, which is translation invariant.  In the process we
significantly lower the threshold $s$ for local well-posedness, though
the current result only applies for small initial data while the
results of \cite{KPV, KPRV1, KPRV2} permit data of arbitrary size.

Our approach is more reminiscent of the preceding result \cite{KPV2}
which established small data local well-posedness for semilinear derivative
Schr\"odinger equations.  Playing a key role is a variant of the
well-known local smoothing estimates which are described below.
The results of \cite{KPV2} also apply in the case that $\Delta$ is
replaced by 
\[\mathcal{L} = \partial_{x_1}^2 + \dots+\partial_{x_k}^2
- \partial_{x_{k+1}}^2 -\dots - \partial^2_{x_n}.\]
The works \cite{Bej1, Bej} work to lower the regularity required in
order to obtain local well-posedness for small initial data.

The smallness hypothesis on the data was removed in \cite{HO}, \cite{Ch},
\cite{KPV3} for the 1 dimensional case, the elliptic case, and the case where $\Delta$
may be replaced by $\mathcal{L}$ respectively.  And \cite{BT} focuses
on improving the necessary regularity.

While some specific models of quasilinear Schr\"odinger equations were
previously studied, the seminal and benchmark results are
\cite{LimPon} in 1-dimension and \cite{KPV}, \cite{KPRV1, KPRV2} in
general dimension.  The interested reader is referred to the more
thorough histories provided in \cite{KPV} and \cite{LinPon}.

The local smoothing estimates, which were mentioned above, were first
established for the Schr\"odinger equation in \cite{CS}, \cite{Sj},
and \cite{Vega} and were motivated by \cite{Kato}, \cite{KrFa} for the
KdV equation.  In particular, we shall use the observation of
\cite{KPV2} that shows that the inhomogeneous estimates provide twice
the smoothing that is available in the homogeneous case.  In the
presence of asymptotically flat operators, such estimates were first
established in \cite{Doi}, \cite{CKS}.


To begin, for each $u$ we denote by $\mathcal{F}u=\hat{u}$ the spatial
Fourier transform of $u$.  We say that the function $u$ is localized
at frequency $2^{i}$ if $\supp \hat{u}(t,\xi)\subset \R\times
[2^{i-1},2^{i+1}]$. Next we introduce a standard Littlewood-Paley
decomposition with respect to spatial frequencies,
\[
1 = \sum_{i=0}^\infty S_i.
\]
Let $\phi_0 : [0,\infty) \to \RR$ be a nonnegative, decreasing, smooth
function such that $\phi_0 (\xi) = 1$ on $[0,1]$ and $\phi_0 (\xi) =
0$ if $\xi \geq 2$.  Then, for each $i \geq 1$ we define $\phi_i :
[0,\infty) \to \RR$ by
$$ \phi_i (\xi) = \phi_0 (2^{-i} \xi) - \phi_0 (2^{-i+1} \xi).$$
We define the operators $S_i$, which localize to frequency
$2^i$, by
\begin{eqnarray*}
 \hat{f}_i (\xi) = \mathcal{F} (S_i f) = \phi_i (\xi) \hat{f} (\xi).
\end{eqnarray*}
We also define the related operators
\begin{eqnarray*}
  S_{\leq N} f = \sum_{i = 0}^N f_i, \ S_{\geq N} f = \sum_{i = N}^\infty f_i.
\end{eqnarray*} 

For each nonnegative integer $j$ we consider a partition $\mathcal{Q}_{j}$ of $\RR^{d}$
into cubes of side length $2^{j}$ and an associated smooth partition of unity
\[
1 = \sum_{Q \in \mathcal{Q}_{j}}  \chi_Q.
\]
Then we define the $l^1_j L^2$ norm by 
\[
\| u\|_{l^1_j L^2} = \sum_{Q \in \mathcal{Q}_{j}} \| \chi_Q u\|_{L^2}.
\]
Our replacement for the $H^s$ initial data space is the space
$l^1 H^s$ with norm given by
\[
\| u\|_{l^1 H^s}^2 = \sum_{j \geq 0} 2^{2sj} \|S_j u\|_{ l^1_j L^2}^2 .
\]
We note that such spaces were previously used in, e.g., \cite{VV}.


The motivation for this choice is as follows. Heuristically
Schr\"odinger waves at frequency $2^j$ travel with speed $2^j$. Hence
on the unit time scale a partition on the $2^j$ spatial scale is
exactly at the threshold where it does not interfere with the linear
flow. In other words, the Schr\"odinger evolution in these spaces at
frequency $2^j$ will be no different from the corresponding evolution
in $H^s$. At the same time, the summability condition with respect to
the $2^j$ spatial scale suffices in order to recover Mizohata's condition 
if $s$ is sufficiently large.
 
As a point of reference, in \cite{BT} similar spaces are defined in
the context of semilinear Schr\"odinger equations. There the
trajectories of the Hamilton flow for the principal part are straight
lines, and one sums $\|f\|_{L^{2}(Q)}$ over those $Q \in \mathcal
Q_j$'s which intersect a line $L \subset \RR^d$ and then take a
supremum with respect to all lines $L$. However, such a definition
relies heavily on the Hamilton flow associated with the Laplacian as the
leading order differential operator.  Here, as we are not guaranteed a
nice Hamilton flow of the leading order operator, we simply sum over all
cubes of scale $2^{j}$.

Our main result concerns the quasilinear problem
\eqref{eqn:quasiquad} with small data $u_0 (x) \in l^1 H^s$:

\begin{thm}
\label{thm:main1}
a) Let $s > \frac{d}2+3$. Then there exists  $\epsilon_0 > 0$
sufficiently small  such that, for all initial data $u_0$ with
$\| u_0 \|_{l^1 H^s} \leq \epsilon_0 $,
the equation \eqref{eqn:quasiquad} is locally well-posed in $l^1 H^s
(\RR^d)$ on the time interval $I = [0,1]$.

b) The same result holds for the equation \eqref{eqn:quasiquad1}
with $s > \frac{d}2 + 2$.
\end{thm}

For comparison purposes we note that the scaling exponent for the
principal part of \eqref{eqn:quasiquad} is $s = \frac{d}2+1$, while
for \eqref{eqn:quasiquad1} with a quadratic nonlinearity in the
gradient $\nabla u$ it is $s = \frac{d}2$.  On the other hand, for the
semilinear version of \eqref{eqn:quasiquad1} the well-posedness result
in $l^1 H^s$ in \cite{Bej}, \cite{BT} applies for $s > \frac{d}2+1$;
that result was shown to be sharp in \cite{Sch}.

We remark that our theorem also holds for the ultrahyperbolic
operators studied in, e.g., \cite{KPRV1, KPRV2}.  Indeed, if $g(0)$ is
of different signature, we need only adjust the local smoothing
estimates of Section \ref{sec:LED}.  The wedge decomposition which is
employed there allows this to be accomplished trivially.

The need to use the $l^1 H^s$ type spaces for the initial data is
exclusively due to the bilinear interactions, both semilinear and
quasilinear. However, we expect these spaces to be relaxed to $H^s$
spaces if all the interactions which are present are cubic and
higher.  Analogs of such observations have appeared previously in
\cite{KPV2, KPV3}.
This problem is considered in a follow-up paper.


For simplicity the life span of the solutions in the above theorem has
been taken to be $[0,1]$.  However, a simple rescaling argument shows
that the life span can be made arbitrarily large by taking
sufficiently small data. By contrast, the short time large data result
cannot be obtained by scaling from the small data result. This is due
to the fact that the spaces used are inhomogeneous Sobolev spaces, and
spatial localization is not allowed due to the infinite speed of
propagation.  In the large data regime, one must also take into account
the existence of trapping.
This problem will also be considered in subsequent work.


The  definition of  ``well-posedness''  in the above theorem is taken
to include the following:

\begin{itemize}
\item Existence of a solution $u \in C([0,T_\epsilon);l^1 H^s)$ satisfying
\[
\| u \|_{L^\infty l^1 H^s} \lesssim \epsilon.
\]

\item Uniqueness in the above class provided that $s$ is large 
enough.

\item Continuity of the solution map 
$$l^1 H^s \ni u_0 \to u \in C([0,T_\epsilon);l^1 H^s)$$ 
for all $s$ as in the theorem.
\end{itemize}

The above conditions allow one to interpret the rough solutions as the
unique limits of smooth solutions. However, in the process of proving
the theorem we introduce a stronger topology $\lX^s \subset 
C([0,T_\epsilon);l^1 H^s)$ and, for all $s$ in the theorem, we 
show that the solutions belong to $\lX^s$, are unique in $\lX^s$ 
and that the solution map $u_0 \to u$ is continuous from
$l^1 H^s$ to $\lX^s$.

We also remark that due to the quasilinear character of the problem
the continuous dependence on the initial data is the best one can hope
for. However, if we assume that the metric $g$ does not depend on $u$,
then the problem becomes semilinear and one obtains
Lipschitz dependence on the initial data as in \cite{BT}.

The paper is organized as follows. In Section \ref{sec:boot}, we
describe the space-time function spaces in which we will solve
\eqref{eqn:quasiquad} \eqref{eqn:quasiquad1}.  In Section
\ref{sec:mult}, we establish the necessary multilinear and nonlinear
estimates in order to close the eventual bootstrap estimates.  In
Section \ref{sec:LED}, we prove the necessary Morawetz type estimate
to establish local energy decay for a linear, inhomogeneous
paradifferential version of the Schr\"odinger equation.  Finally, in
Section \ref{sec:proof}, we combine the above estimates with the
proper paradifferential decomposition of the equation in order to
prove Theorem \ref{thm:main1}.

{\sc Acknowledgments.} {The first author was supported in part by an
NSF Postdoctoral Fellowship and wishes to thank the Courant Institute
for generously hosting him during part of the proof of this result.
The second author is supported in part by NSF grants DMS-0800678 and DMS-1054289.  
The third author is supported in part by  NSF grant DMS0354539
as well as by the Miller Foundation.}

\section{Function Spaces and Notations}
\label{sec:boot}

\subsection{ The $l^p_j U$ spaces }
As a generalization of the $l^1_j L^2$ norm defined in the
introduction, given any translation invariant Sobolev type space $U$
we define the Banach spaces $l^p_j U$ with norm
\begin{eqnarray*}
\| u \|_{l^p_j U}^p =  \sum_{Q \in \mathcal{Q}_j} \| \chi_Q u
  \|_U^p  
\end{eqnarray*}
with the obvious changes when $p = \infty$. By a slight abuse we will
employ the same notation whether $U$ represents a space-time Sobolev
space or a purely spatial Sobolev space. Note that in what follows we
will work with inhomogeneous norms, so we take only cubes of size $1$
or larger, i.e. $j \geq 0$.  In particular we will use the dual space
$l^\infty_j L^2$ to $l^1_j L^2$, with norm
\[
\| u\|_{l^\infty_j L^2} = \sup_{Q \in \mathcal{Q}_{j}} \| \chi_Q u\|_{L^2} .
\]

By replacing the sum over $Q$ above with an integral, one can easily
see that these spaces admit a translation invariant equivalent norm.

We also note that the smooth partition of compactly supported cutoffs
in the $l^1_j U$ spaces can be replaced by cutoffs which are frequency
localized.  Indeed, we have that
\begin{eqnarray}
\label{eqn:bernstein}
\sum_{Q\in \Q_j} \|\chi_Q u\|_U \approx \sum_{Q\in \Q_j}
\|(S_{0}\chi_Q) u\|_U.
\end{eqnarray}
This follows simply from the fact that $S_{0}\chi_Q$ decays rapidly
away from $Q$. We will use frequency localized cutoffs whenever we
need the components $\chi_Q u$ to retain the frequency localization of
$u$.

\subsection{ The $X$ and $Y$ spaces}
We next define a local energy type space $X$
of functions on $[0,1] \times \R^d$  with norm
\[
  \| u \|_{X}  =   \sup_{l} \sup_{Q \in \mathcal{Q}_l} 
2^{-\frac{l}{2}} \| u \|_{L^2_{t,x} ([0,1] \times Q)}.
\]

To measure the right hand side of the Schr\"odinger equation we
use a dual local energy space $Y \subset L^2([0,1] \times \R^d)$, which as we
will show satisfies the duality relation $X=Y^*$.  The $X$ spaces are
time-adapted Morrey-Campanato spaces, and for the relation to $Y$,
see, e.g., \cite{BRV}.


The space $Y$ is an atomic space.  A function $a$ is an atom in $Y$ if
there exists some $j\ge 0$ and some cube $Q \in \mathcal Q_j$ so that
$a$ is supported in $[0,1] \times Q$ and
\[
\| a\|_{L^2([0,1] \times Q)} \lesssim 2^{-\frac{j}2}.
\]
The space $Y$ is the Banach space of linear combinations of the form
\begin{equation}\label{sumofatoms}
f = \sum_k c_k a_k, \qquad  \sum |c_k| < \infty, \qquad a_k \text{ atoms } 
\end{equation}
 with respect to the norm
\[
\| f \|_{Y} = \inf \{\sum |c_k|\,:\, f=\sum_k c_k a_k,\, a_k \text{ atoms} \}.
\]
The core spaces $X$, $Y$ are related via the following duality relation.
\begin{prop}
  The following duality relation holds with respect to the standard $L^2$ 
duality: $Y^* = X$.
\end{prop}

\begin{proof}

It is clear by construction that
  \begin{eqnarray*}
    (u,v)_{t,x} \lesssim \| u \|_{X} \| v \|_{Y}.
  \end{eqnarray*}
  Hence, we need to show for any $L \in Y^*$, there exists $u
  \in X$ such that
  \begin{eqnarray*}
    (u,v)_{t,x} = L(v), \quad \| u \|_{X} \leq \| L \|_{Y^*}.
  \end{eqnarray*}
Applying $L$ to all atoms associated to a cube  $Q \in \mathcal{Q}_j$, 
we obtain
\[
|Lv| \lesssim 2^{\frac{j}2} \|L\|_{Y^*} \|v\|_{L^2}
\]
for all functions $v \in L^2$ with support in $Q$.
Hence by Riesz's theorem there exists a function $u_Q$ in $Q$ 
so that
\[
Lv = \langle u_Q,v \rangle, \qquad \|u_Q\|_{L^2} \lesssim  2^{\frac{j}2} \|L\|_{Y^*}.
\]
A priori the functions $u_Q$ depend on $Q$. However, given two intersecting cubes
$Q_1$ and $Q_2$, the actions of $u_{Q_1}$ and $u_{Q_2}$ must coincide 
as $L^2$ functions in $Q_1 \cap Q_2$. Hence we must have $u_{Q_1} = u_{Q_2}$
on $Q_1 \cap Q_2$.  Thus there is a single global function $u$ so that, for 
every cube $Q$,  $u_Q$ is the restriction of $u$ to $Q$. Then the  last estimate
shows that 
\[
\| \chi_Q u\|_{L^2} \lesssim  2^{\frac{j}2} \|L\|_{Y^*}, \qquad Q \in \mathcal{Q}_j
\]   
or equivalently
\[
\|u\|_{X} \lesssim   \|L\|_{Y^*}.
\]

\end{proof}

\subsection{ The $\lX^s$ and $\lY^s$ spaces }

We first remark that the $X$ norm corresponds exactly 
to the local energy decay estimates for $H^{-\frac12}$ 
solutions to the Schr\"odinger equation. Precisely,
in the constant coefficient case we have the following dyadic bound
\[
\| e^{it\Delta} S_j f\|_{X} \lesssim 2^{-\frac{j}2} \|f\|_{L^2}.
\]
Thus for $L^2$ solutions to the linear Schr\"odinger equations
which are localized at frequency $2^j$  it is natural to use 
the space 
\[
X_j = 2^{-\frac{j}2} X \cap L^\infty L^2
\]
with norm
\[
\| u\|_{X_j} = 2^{\frac{j}2} \|u\|_{X} 
+ \|u\|_{L^\infty L^2}.
\]
Adding the $l^1$ spatial summation on the $2^j$ scale
we obtain the space $l^1_j X_j$ with norm
\begin{eqnarray*}
  \| u \|_{l^1_j X_j} = 
\sum_{Q \in \mathcal{Q}_j} \| \chi_Q u \|_{X_j}.
\end{eqnarray*}

Then  we define the space $\lX^s$ where we seek solutions to the
nonlinear Schr\"odinger equations \eqref{eqn:quasiquad},
\eqref{eqn:quasiquad1} with $l^1 H^s$ data by
\begin{eqnarray*}
  \| u \|^2_{\lX^s}  =  \sum_j 2^{2js} \| S_{j} u \|^2_{l^1_j X_j}.
\end{eqnarray*}


The appropriate space for the inhomogeneous term for $L^2$ solutions
to the Schr\"odinger equation at frequency $2^j$ is 
\[
Y_j = 2^{\frac{j}2} Y + L^1 L^2
\]
with norm
\[
\| f \|_{Y_j} = \inf_{f =  2^{\frac{j}2} f_1 + f_2} \|f_1\|_{Y} + 
\|f_2\|_{L^1 L^2}.
\]
 To fit it to the context in the present paper we add the
$l^1_j$ summation and work with the space $l^1_j Y_j$.   
Finally, we define the space $\lY^s$ with norm
\begin{eqnarray*}
  \| f \|^2_{\lY^s} =  \sum_j  2^{2js} \| S_{j} f \|^2_{l^1_j Y_j} .
\end{eqnarray*}

\subsection{Frequency envelopes}

For both technical and expository reasons it is convenient to present
our bilinear and nonlinear estimates using the method of frequency
envelopes. Given a Sobolev type space $U$ so that
\[
\|u\|_{U}^2 \sim \sum_{k=0}^\infty \|S_k u\|_{U}^2 
\]
a frequency envelope for $u$ in $U$ is a positive sequence $a_j$ so
that
\begin{equation} 
\| S_j u\|_{U} \leq a_j \|u\|_{U}, \qquad 
 \sum a_j^2 \approx 1.
\end{equation}
We say that a frequency envelope is admissible if $a_0\approx 1$ and 
it is slowly varying,
\[
a_j \leq 2^{\delta |j-k|} a_k, \qquad j,k \geq 0 , \qquad 0 < \delta
\ll 1.
\]
An admissible frequency envelope 
always exists, say by 
\begin{equation}\label{freqEnv}
a_j = 2^{-\delta j} + \|u\|^{-1}_{U} \max_k 2^{-\delta |j-k|} \|S_k u\|_{U}.
\end{equation}

In the sequel we will use frequency envelopes for the spaces $l^1
H^s$, $\lX^s$ and $\lY^s$.  The parameter $\delta$ is a 
sufficiently small parameter, which will only depend on the value of $s$
in our main theorem. For instance in the case of part (b) of the theorem,
we will choose $\delta$ so that
\begin{eqnarray*}
0< \delta < s-\frac{d}{2}-2.
\end{eqnarray*}


\section{Multilinear and nonlinear estimates}
\label{sec:mult}

In this section we prove the main bilinear and nonlinear estimates
in the paper.  We begin with a shorter proposition containing our bilinear and 
Moser estimates in terms of the $\lX^s$ and $\lY^s$ spaces.

\begin{prop}
\label{prop:nonlinear}
We have the following:

\noindent a) Let $s > \frac{d}2$.  Then the $\lX^s$ spaces satisfy the
algebra property
  \begin{equation}\label{u_squared}
\| u v\|_{\lX^s} \lesssim \|u\|_{\lX^s} \|v\|_{\lX^s}, 
  \end{equation}
as well as the Moser estimate
  \begin{equation}\label{moser}
\| F(u) \|_{\lX^s} \lesssim \|u\|_{\lX^s}(1+\|u\|_{\lX^s}) c(\|u\|_{L^\infty}).  
  \end{equation}
for all smooth $F$ with $F(0)= 0$.

\noindent b) Bilinear $X \cdot X \to Y$ bounds. Let $s >    \frac{d}2+2$. Then 
 \begin{equation}\label{xxy2}
\| u v\|_{\lY^\sigma} \lesssim \|u\|_{\lX^{s-1}} \|v\|_{\lX^{\sigma-1}}, \qquad 
\ \ 0 \leq \sigma \leq s,
  \end{equation}
 \begin{equation}\label{xxy1}
\| u v\|_{\lY^\sigma} \lesssim \|u\|_{\lX^{s-2}} \|v\|_{\lX^{\sigma}}, \qquad 
\ \ 0 \leq \sigma \leq s-1.
  \end{equation}
\end{prop}

The estimates in the above proposition suffice for most of our purposes,
but not all. Instead we need a sharper version of it, which is phrased
in terms of frequency envelopes. Thus Proposition~\ref{prop:nonlinear}
is a direct consequence of the next proposition:

\begin{prop}
\label{fe:nonlinear}
We have the following:

\noindent a) Let $s > \frac{d}2$, and $u,v \in \lX^s$ with admissible
frequency envelopes
$a_k$, respectively $b_k$. Then the $\lX^s$ spaces satisfy the
algebra type property
  \begin{equation}\label{fe:u_squared}
\| S_k (u v)\|_{\lX^s} \lesssim (a_k+b_k)\|u\|_{\lX^s} \|v\|_{\lX^s}, 
  \end{equation}
as well as the Moser type estimate
  \begin{equation}\label{fe:moser}
\|S_k F(u) \|_{\lX^s} \lesssim a_k \|u\|_{\lX^s}(1+\|u\|_{\lX^s}) c(\|u\|_{L^\infty}). 
  \end{equation}
for all smooth $F$ with $F(0)= 0$.

\noindent
b) Bilinear $X \cdot X \to Y$ bounds. Let $s > \frac{d}2+2$, $\sigma \leq s$
and $u \in \lX^s$, $v \in \lX^\sigma$ with admissible frequency envelopes
$a_k$, respectively $b_k$. Then 
\begin{equation}\label{fe:xxy2}
\| S_k(u v)\|_{\lY^\sigma} \lesssim (a_k + b_k)
\|u\|_{\lX^{s-1}} \|v\|_{\lX^{\sigma-1}}, \qquad 
\ \ 0 \leq \sigma \leq s,
\end{equation}
\begin{equation}\label{fe:xxy1}
\|S_k( u v) \|_{\lY^\sigma} \lesssim (a_k + b_k)
\|u\|_{\lX^{s-2}} \|v\|_{\lX^{\sigma}}, \qquad 
\ \ 0 \leq \sigma \leq s-1 ,
\end{equation}
\begin{equation}\label{fe:xxy3}
\| S_k(u S_{\geq k-4}v)\|_{\lY^\sigma} \lesssim (a_k + b_k)
\|u\|_{\lX^{s-2}} \|v\|_{\lX^{\sigma}}, \qquad 
\ \ 0 \leq \sigma \leq s .
\end{equation}

c) Commutator bound. For $s > \frac{d}2+2$ and any multiplier $A \in S^0$ we have
\begin{equation}\label{fe:xxycom}
\| \nabla [S_{< k-4} g,A(D)] \nabla S_k u  \|_{\lY^0} \lesssim 
 \|g - I \|_{\lX^{s}} \|S_k u\|_{\lX^0}.
\end{equation}
\end{prop}

\begin{proof}
 A preliminary step in the proof is to observe that we 
have a Bernstein type inequality,
\[
\| S_k u\|_{l^1_k L^\infty} \lesssim 2^{\frac{dk}2} \| S_k u\|_{l^1_k L^\infty L^2} 
\lesssim 2^{\frac{dk}{2}}\| S_k u\|_{l^1_k X_k}.
\]
This is easily proved using the classical Bernstein inequality, 
with frequency localized cube cutoffs.
After dyadic summation this gives
\begin{equation}\label{eqn:bern}
 \|u\|_{L^\infty} \lesssim \|u\|_{\lX^s}, \qquad s>d/2,
\end{equation}
respectively
\begin{equation}\label{eqn:bern2}
 \|S_{<j} u\|_{l^1_j L^\infty} \lesssim \|u\|_{\lX^s}, \qquad s>d/2.
\end{equation}

To prove the $X$ algebra property we consider 
the usual Littlewood-Paley dichotomy. In a dyadic expression
$S_k( S_i u S_j v)$ we need to consider two cases: 

{ \em High-low interactions:}
$j < i-4$ and $|i-k| < 4$ (or the symmetric alternative). Then 
the $l^1_k X_k$ and $l^1_i X_i$ norms are comparable therefore
we have
\[
\| S_i u  S_j v\|_{l^1_k X_k} \lesssim \| S_i u\|_{l^1_i X_i}
\|  S_j v\|_{L^\infty} \lesssim  2^{\frac{dj}2} \| S_i u\|_{l^1_i X_i} 
\| S_j v\|_{L^\infty L^2}.
\]
The multiplier $S_k$ is bounded in $l^1_k X_k$, therefore we obtain
\[
\|S_k( S_i u  S_j v)\|_{\lX^s} \lesssim  2^{(\frac{d}2-s)j} a_i b_j \|u\|_{\lX^s}
\|v\|_{\lX^s}.
\] 
Upon summation over $i, j$, we get the desired bound for the high-low
interactions.

{ \em High-high interactions:} $|i-j| \leq 4$ and $i,j \geq k-4$.
For $j > k$ we use Bernstein's inequality at frequency $2^k$ to obtain
\[
\begin{split}
\| S_k(S_i u  S_j v)\|_{l^1_k X_k} \lesssim  2^{\frac{kd}2}  
\| S_i u\|_{l^1_k X_k}  \|S_j v \|_{L^\infty L^2}.
\end{split}
\]

Each $\mathcal Q_i$ cube contains about $2^{d(i-k)}$ 
$\mathcal Q_k$ cubes and $X_i \subset X_k$; therefore we obtain
\[
\| S_k (S_i u  S_j v) \|_{l^1_k X_k} \lesssim  2^{d(i-k)} 2^{\frac{kd}2} 
\| S_i u\|_{l^1_i X_i}  \| S_j v \|_{L^\infty L^2} ,
\]
i.e.
\begin{equation}\label{partAref}
\| S_k(S_i u  S_j v) \|_{\lX^s} \lesssim  2^{(\frac{d}2-s)(2i-k)}  a_i b_j
\|u\|_{\lX^s}  \|v \|_{\lX^s} .
\end{equation}
The corresponding part of the bound \eqref{fe:u_squared} follows after summation over $i,j$.

Next we turn our attention to the Moser estimate \eqref{fe:moser}.
Following an idea in \cite{T-WM} we consider a multilinear
paradifferential expansion, which follows from the Fundamental Theorem
of Calculus. For the purpose of this proof we replace the discrete
Littlewood-Paley decomposition by a continuous one
\[
Id = S_0 + \int_{0}^\infty S_k \ dk,
\]
denote $u_k = S_k u$, and, by a slight abuse of notation,  $\uz=S_0 u$.
Then we can write
\begin{equation}
  \label{ftc}
S_k F(u) = S_k F(\uz) + \int_0^\infty S_k(u_{k_1}F'(u_{<k_1}))\,dk_1.
\end{equation}

To estimate the first term, we begin with 
\[\|\partial^\alpha \uz\|_{L^\infty}\lesssim
\|\uz\|_{L^\infty},\quad \|\partial^\alpha
\uz\|_{l^1_0X}\lesssim \|\uz\|_{l^1_0X}.\]
Then, repeated applications of the chain rule lead to 
\begin{align*}
  \|\partial^\alpha F(\uz)\|_{L^\infty} &\lesssim
  \|\uz\|_{L^\infty} c(\|\uz\|_{L^\infty}),\\
\|\partial^\alpha F(\uz)\|_{l^1_0X_0}&\lesssim \|\uz\|_{l^1_0X_0} c(\|\uz\|_{L^\infty}).
\end{align*}
Hence
\[
\|S_k F(\uz)\|_{l^1_k X_k} \lesssim 2^{\frac{k}2}\|S_k F(\uz)\|_{l^1_0
  X_0}\lesssim 2^{-Nk} \|\uz\|_{l^1_0 X_0}c(\|\uz\|_{L^\infty})
\]
for any $N$.  The $\lX^s$ bound for the first term of \eqref{ftc} then
follows trivially.

For the second term in \eqref{ftc}, we consider three cases.

{\bf Case I:}  $k-4\le k_1 \le k+4$.  This is the easiest case as
\[
\|S_k(u_{k_1} F'(u_{<k_1}))\|_{l^1_k X_k}\lesssim
\|u_{k_1}\|_{l^1_{k_1} X_{k_1}} c(\|u_{<k_1}\|_{L^\infty}),
\]
therefore
\[
\|S_k(u_{k_1} F'(u_{<k_1}))\|_{\lX^s} \lesssim a_{k_1} \|u\|_{\lX^s}   c(\|u_{<k_1}\|_{L^\infty}).
\]
For $|k-k_1|\leq 4$ we have $a_{k_1} \sim a_k$, and the $k_1$ integration is trivial.

{\bf Case II:}  $k_1 < k-4$.  In this case,
\[
S_k(u_{k_1} F'(u_{<k_1})) = S_k(u_{k_1} \tilde{S}_{k} F'(u_{<k_1})),
\]
for a multiplier $\tilde{S}_k$ which similarly localizes to frequency $2^k$ and 
\[
S_k\tilde{S}_k = S_k.
\]
Applying the chain rule as above, it follows that
\begin{equation}
\|\tilde{S}_k F'(u_{<k_1})\|_{L^\infty} \lesssim 2^{-N(k-k_1)}
c(\|u_{<k_1}\|_{L^\infty}), \qquad k_1 \leq k
\label{gli}
\end{equation}
and thus,
\begin{align*}
  \|S_k (u_{k_1}F'(u_{<k_1}))\|_{l^1_k X_k} &\lesssim
  2^{\frac{k-k_1}2} \|u_{k_1}\|_{l^1_{k_1} X_{k_1}} \|\tilde{S}_k F'(u_{<k_1})\|_{L^\infty}\\
&\lesssim 2^{-N(k-k_1)} \|u_{k_1}\|_{l^1_{k_1}X_{k_1}} c(\|u_{<k_1}\|_{L^\infty}),
\end{align*}
which leads to
\[
  \|S_k (u_{k_1}F'(u_{<k_1}))\|_{\lX^s} \lesssim   2^{-N(k-k_1)} a_{k_1} \|u\|_{\lX^s}
c(\|u\|_{L^\infty}).
\]
The $k_1$ integration is now straightforward.

 {\bf Case III:} $k_1 > k+4$.  
In this case, we apply the Fundamental Theorem of Calculus again to
see that
\begin{multline}\label{orderTwo}
\int_{k+4}^\infty S_k(u_{k_1}F'(u_{<k_1}))\,dk_1 = \int_{k+4}^\infty
S_k(u_{k_1} F'(\uz))\,dk_1 \\+ \int_{k+4}^\infty \int_0^{k_1}
S_k(u_{k_1}u_{k_2}F''(u_{<k_2}))\ dk_2\ dk_1.
\end{multline}

For the first term in the right of \eqref{orderTwo}, we have that
\[
S_k(u_{k_1}F'(\uz)) = S_k(u_{k_1}\tilde{S}_{k_1}F'(\uz)).
\]
Therefore, as there are $2^{d(k_1-k)}$ cubes of sidelength $2^k$
contained in a cube with sidelength $2^{k_1}$, it follows that
\begin{align*}
  \|S_k(u_{k_1}F'(\uz))\|_{l^1_k X_k} & \lesssim 2^{d(k_1-k)}
  \|u_{k_1}\|_{l^1_{k_1}X_{k_1}} \|\tilde{S}_{k_1}F'(\uz)\|_{L^\infty}\\
&\lesssim 2^{d(k_1-k)-Nk_1} \|u_{k_1}\|_{l^1_{k_1}X_{k_1}}
\|\uz\|_{L^\infty} c(\|\uz\|_{L^\infty}).
\end{align*}
This yields
\[
 \|S_k(u_{k_1}F'(\uz))\|_{\lX^s} \lesssim 2^{(d-s)(k_1-k)} 2^{-Nk_1} a_{k_1} \|u\|_{\lX^s} 
c(\|\uz\|_{L^\infty}).
\]
The desired estimate follows easily after  a $k_1$  integration. 

We now examine the second term in the right of \eqref{orderTwo}.  Here
we have two subcases to examine separately. 

{\bf Case III(a):}  $k_1-4 \leq k_2 \leq  k_1$.
Then we can argue as in \eqref{partAref} to obtain
\[\|S_k(u_{k_1}u_{k_2}F''(u_{<k_2}))\|_{l^1_kX_k}\lesssim 2^{dk_1}
2^{-\frac{dk}{2}} \|u_{k_1}\|_{l^1_{k_1}X_{k_1}} \|u_{k_2}\|_{L^\infty L^2}
c(\|u\|_{L^\infty}).
\]

{\bf Case III(b):} $0 < k_2 \leq k_1-4$. Then
\[
S_k(u_{k_1}u_{k_2}F''(u_{<k_2})) =
S_k(u_{k_1}u_{k_2}\tilde{S}_{k_1}F''(u_{<k_2})).
\]
Therefore using \eqref{gli} for $\tilde{S}_{k_1}F''(u_{<k_2})$ and
Bernstein's inequality at frequency $2^k$ we have
\begin{multline*}
  \|S_k(u_{k_1}u_{k_2}F''(u_{<k_2}))\|_{l^1_kX_k}\\\lesssim \ 2^{d(k_1-k)}2^{\frac{dk}2}
2^{-N(k_1-k_2)}
  \|u_{k_1}\|_{l^1_{k_1}X_{k_1}} \|u_{k_2}\|_{L^\infty L^2} c(\|u\|_{L^\infty}).
\end{multline*}

Combining the two cases and adding in the Sobolev weights this leads to
\[
\begin{split}
\|S_k(u_{k_1}u_{k_2}F''(u_{<k_2}))\|_{\lX^s} \lesssim  2^{(2k_1-k)(\frac{d}2-s)-N(k_1-k_2)}
a_{k_1} a_{k_2} \| u\|_{\lX^s}^2 c(\|u\|_{L^\infty})
\end{split}
\]
which can be integrated with respect to $k_1$, $k_2$.

 b) As a general rule, here we always estimate the bilinear expressions in $Y$,
and never in $L^1 L^2$. By the definition of the $Y$ space, for each $l \le j$ 
we have 
\begin{equation}\label{easy-y}
\| f\|_{l^1_j Y} \lesssim 2^{\frac{l}2} \|f\|_{l^1_l L^2}.
\end{equation}
We use the standard Littlewood-Paley dichotomy, and consider expressions of the 
form $S_k(S_i u S_j v)$. There are two cases to examine.

{\em High-low interactions:} $|i-k| \leq 4$ and $j < i-4$. 
Applying \eqref{easy-y} with $l = j$ we obtain
\[
\| S_{i} u S_j v \|_{l^1_k Y_k } \lesssim
2^{\frac{j-k}2} 
\| S_{i} u S_j v\|_{l^1_j L^2}
\lesssim  2^{\frac{j-k}2}  \|S_i u\|_{l^\infty_j L^2}
\| S_{j} v\|_{l^1_j L^\infty}.
\]
For the first factor we  use the $X$ norm  and for the second
we use Bernstein's inequality. This yields
\begin{equation}\label{bilinNoSum}
\| S_{i} u S_j v \|_{l^1_k Y_k} \lesssim 
 2^{\frac{d+2}{2}j -k }   \|S_i u\|_{X_i}
\| S_{j} v\|_{l^1_j L^\infty L^2}, 
\end{equation}
and further
\begin{equation}
\|S_k(S_{i} u S_j v)\|_{l^1_k Y_k} \lesssim 
 2^{\frac{d+2}{2}j -k }   \|S_i u\|_{l^1_i X_i}
\| S_{j} v\|_{l^1_j X_j}. 
\label{xxy:hl}\end{equation}

The alternative low-high interactions can be handled by similar arguments. 

{\em High-High interactions.}  $|i-j| \leq 4$ and $i,j \geq k-4$.
Applying \eqref{easy-y} with $l=k$,
Cauchy-Schwarz to transition from $2^k$ sized cubes to $2^j$ sized
cubes and then Bernstein's inequality, we have
\begin{align*}
  \| S_k (S_{i} u S_{j} v)\|_{l^1_k Y_k} &\lesssim  
\| S_k(S_{i} u S_j v)\|_{l^1_k L^2}\\
&\lesssim  2^{\frac{d}2(j-k)}
\| S_k(S_{i} u S_j v)\|_{l^1_j L^2}\\
&\lesssim 2^{\frac{jd}{2}} \|S_k(S_iuS_jv)\|_{l^1_j
  L^2_t L^1_x}\\
&\lesssim 2^{\frac{jd}{2}} \|S_i u\|_{l^1_i L^2} \|S_j
v\|_{L^\infty L^2}.
\end{align*}
Thus we obtain
\begin{equation}
\| S_k(S_{i} u S_j v) \|_{l^1_k Y_k} \lesssim 2^{\frac{jd}{2}}
 \|S_i v\|_{l^1_i X_i}
\| S_{j} u\|_{l^1_j X_j}.  
\label{xxy:hh}
\end{equation}

The desired bounds  \eqref{fe:xxy1}, \eqref{fe:xxy2} 
and \eqref{fe:xxy3} follow easily from the dyadic bounds \eqref{xxy:hl}
and \eqref{xxy:hh} after summation.

c)  For the commutator we claim the representation
\begin{equation}
\nabla [S_{< k-4} g,A(D)] \nabla S_k u  = L( \nabla S_{< k-4} g,  \nabla S_k u )
\label{comrep}\end{equation}
where $L$ is a disposable operator, i.e. a translation invariant operator
of the form 
\[
L(f,g)(x) = \int f(x+y) g(x+z) w(y,z) dy dz, \qquad \|w\|_{L^1} \lesssim 1.
\]
Assume this representation holds. Then, since the $\lX^s$ spaces are
translation invariant (i.e. they admit translation invariant
equivalent norms), the commutator bound \eqref{fe:xxycom} becomes a
direct consequence of \eqref{fe:xxy2}.

To prove \eqref{comrep} we first observe that we can harmlessly replace
the multiplier $A(D)$ by $\tilde S_k A(D)$ and $S_{< k-4} g$ by 
$\tilde S_{< k-4} S_{< k-4} g$. Denoting $g_1 = S_{< k-4} g$ and $ u_1 =  \nabla S_k u$,
the above commutator is written in the form
\[
C(g_1,u_1) = \nabla [\tilde S_{< k-4} g_1,\tilde S_k A(D)] u_1.
\]
The operators $\tilde S_k A(D)$ and $\tilde S_{< k-4}$
have  kernels $K(y)$, $H(y)$ which
satisfies bounds of the form
\[
|\partial^\alpha K(y)|, |\partial^\alpha H(y)|\lesssim_\alpha 2^{(d+|\alpha|)k} (1+ 2^k |y|)^{-N}. 
\]
Then we can write
\[
\begin{split}
C(g_1,u_1) (x)= &  \nabla_x \int (g_1(x-z) - g_1(x-y-z)) H(z) K(y) u_1(x-y) dy dz 
\\
= &  \nabla_x \int_0^1 \int  y \nabla g_1(x-z - hy) H(z) K(y) u_1(x-y) dy dz dh 
\\
= &  \nabla_x \int_0^1 \int  y \nabla g_1(x-z) H(z+hy) K(y) u_1(x-y) dy dz dh 
\end{split}
\]
Distributing the $x$ derivative in front and integrating by parts with respect to either
$y$ or $z$ this leads to the representation  \eqref{comrep} where the kernel $w$ of $L$
is given by 
\[
w(y,z) = (\nabla_y + \nabla_z)  \int_0^1 y H(z+hy) K(y) dh.
\]
The $L^1$ bound on $w$ follows from the above bounds on $H$, $K$ and their 
derivatives.
\end{proof}


\section{Local Smoothing Estimates}
\label{sec:LED}

In this section we consider a frequency localized linear Schr\"odinger 
equation
\begin{equation} \label{freq-loc}
(i \partial_t + \partial_k g^{kl}_{< j-4}  \partial_l ) u_j = f_j, \qquad 
u_j(0)= u_{0j} .
\end{equation}
The main result of this section is as follows:
\begin{prop}
\label{prop:morawetz}
Assume that the coefficients $g^{kl}$ in  \eqref{freq-loc} 
satisfy
\begin{equation}\label{small_hyp}
\| g^{kl}-\delta^{kl}\|_{\lX^s} \ll 1
\end{equation}
for some $s>\frac{d}{2}+2$.  
Let   $u_j$  be a solution to  \eqref{freq-loc}  which
is localized at frequency $2^j$. Then the following estimate 
holds:
\begin{equation} \label{l1le}
 \| u_{j} \|_{l^1_j X_j} \lesssim \| u_{0j} \|_{l^1_j L^2} +  
 \| f_{j} \|_{l^1_j Y_j} .
\end{equation}
\end{prop}

\begin{proof}
Dropping the $l^1_j$ summation, our main task will be to prove the 
simpler bound 
\begin{equation} \label{le}
 \| u_{j} \|_{X_j } \lesssim \| u_{0j} \|_{L^2} +  
 \| f_{j} \|_{Y_j}.
\end{equation}
Then \eqref{l1le} will follow easily via $\mathcal Q_j$ localizations.
We rewrite the equation \eqref{freq-loc} in the form
\[
(i \partial_t - A ) u_j = f_{j1} + f_{j2}, \qquad 
u_j(0)= u_{0j},
\]
where $A = -\partial_k g^{kl}_{< j-4}  \partial_l $ is self-adjoint and
$f_{j1} \in L^1 L^2$,  $f_{j2} \in Y$.

The estimate \eqref{le} has two components, an energy bound and local
energy decay.  We have the trivial inequality $\|u\|_{X} \lesssim \|u\|_{L^\infty L^2}$;
therefore the energy estimate suffices for small $j$.

The energy-type estimate is standard if the right hand
side is in $L^1_tL^2_x$, but we would like to allow the right hand
side to be in the dual smoothing space as well.  Using the common
notation $D_t = \frac{1}{i}\partial_t$, we frame it in an
abstract framework as follows:

\begin{lemma}
Let $A$ be a self-adjoint operator. Let $u$ solve the equation 
\begin{equation}
(D_t +A) u = f \qquad u(0) = u_0
\end{equation}
in the time interval $[0,T]$. Then we have 
\begin{equation}
\| u\|_{L^\infty_t L^2_x}^2 \lesssim \| u_0\|_{L^2}^2
+ \| u\|_{X_j} \|f\|_{Y_j}. 
\label{eest}\end{equation}
\end{lemma}

\begin{proof}
We need only compute
\begin{equation}\label{eest2}
\frac{d}{dt} \frac12 \|u(t)\|_{L^2}^2 = \Im \langle u, f \rangle,
\end{equation}
and notice that for each $t \in [0,T]$ we have by duality
\[
\| u(t)\|_{L^2}^2 \lesssim \|u(0)\|_{L^2}^2 + \| u\|_{X_j} \|f\|_{Y_j}.
\]
We take the supremum over $t$ on the left and  the conclusion follows.
\end{proof}

Next we consider the local energy decay estimate.
We will prove that the following holds for $Q \in \mathcal {Q}_{l}$ and $ 0 \leq l \leq j$:
\begin{equation}\label{le-l}
\begin{split}
  2^{j-l}\| u_j\|_{L^2(Q)}^2 \lesssim & \ 
\|u_j\|_{L^\infty L^2}^2  +
  \|u_j\|_{X_j} \|f_j\|_{Y_j}  
  + (2^{-j} +  \| g - I\|_{\lX^s}) \|u_j\|_{X_j}^2.
\end{split}
\end{equation}
Suppose this is true. Taking the supremum over $Q \in \mathcal {Q}_{l}$
and over $l$,
we obtain 
\[
\begin{split}
  2^{j}\| u_j\|_{X}^2 \lesssim & \
\|u_j\|_{L^\infty L^2}^2+
  \|u_j\|_{X_j} \|f_{j}\|_{Y_j}    + (2^{-j}+ \| g - I\|_{\lX^s}) \|u_j\|_{X_j}^2.
\end{split}
\]
The last term on the right can be discarded 
for large enough $j$ since $\| g - I\|_{\lX^s} \ll 1$. Then we obtain
\[
\begin{split}
  2^{j}\| u_j\|_{X}^2 \lesssim 
\|u_j\|_{L^\infty L^2}^2 +
  \|u_j\|_{X_j} \|f_{j}\|_{Y_j}.
\end{split}
\]
Combined with \eqref{eest} this gives \eqref{le} by the Cauchy-Schwarz inequality.

We now turn our attention to the proof of \eqref{le-l}. 
For a self-adjoint multiplier $\M$, we have
\begin{equation}\label{abstractcomm}
\frac{d}{dt}\la u,\M u\ra = -2\Im\la(D_t+A)u,\M u\ra + \la
i[A,\M]u,u\ra.
\end{equation}
We then wish to construct $\M$ so that
\begin{enumerate}
\item $\|\M u\|_{L^2_x}\lesssim \|u\|_{L^2_x}$,
\item $\|\M u\|_X \lesssim \|u\|_X$,
\item $i\la [A,\M]u,u\ra \gtrsim 2^{j-\ell} \|u\|_{L^2_{t,x}([0,1]\times Q)}^2
- O(2^{-j} +\| g-I\|_{\lX^s}) \| u\|_{X_j}^2$.
\end{enumerate}
If these three properties hold for $u = u_j$ and $ (D_t+A)u_j = f_{j}$, 
then the bound \eqref{le-l} follows. 

As a general rule, we will choose $\M$ to be a first order differential operator 
with smooth coefficients localized at frequency $\lesssim 1$,
\begin{equation}
i 2^j \M = a^k(x) \partial_k + \partial_k a^k(x)
\end{equation}

A key step in our analysis is to dispense with the contribution 
of the difference $g-I$ in the commutator $[A,\M]$. Precisely,
we have  

\begin{lemma}
Let $A = \partial_k g^{kl}  \partial_l $ with $g = g_{< j-4}$
and $\M$ be as above.  Suppose that $s>\frac{d}{2}+2$. 
Then we have 
\begin{equation}
| \la [A,\M]u_j,u_j\ra| \lesssim  \| g - I \|_{\lX^s} \|u_j\|_{X_j}^2.
\end{equation}
\end{lemma}
\begin{proof}
The commutator  $[A,\M]$ can be written in the form
\[
i[A,\M] = 2^{-j} ( \nabla ( g \nabla a + a \nabla g) \nabla + \nabla g \nabla^2 a +
g \nabla^3 a ).
\]
All the $a$ factors are bounded and low frequency, and can therefore
trivially be discarded. Hence the worst term we need to estimate is 
\[
2^{-j} \la a \nabla g \nabla u_j, \nabla u_j \ra.
\]
Due to the frequency localization of $u_j$ we have 
\[
\| \nabla u_j \|_{X_j} \lesssim  2^j \| u_j\|_{X_j}.
\]
Hence by the $Y_j^*=X_j$ duality it remains to show that 
\[
\|  (\nabla g_{<j-4}) v_j\|_{Y_j} \lesssim 2^{-j} \|g - I \|_{\lX^s}\|v_j\|_{X_j}
\]
for $v_j = \nabla u_j$. But this is a consequence of the bilinear bound
\eqref{bilinNoSum}.
\end{proof}


The next step is to prove \eqref{le-l} under the additional assumption that
$u_j$ is frequency localized in an angle
\begin{equation}
\supp \hat u_j \subset \{ |\xi| \lesssim \xi_1 \}.
\label{uj-ang}\end{equation}
Here, we take a small angle about the first coordinate axis, and the
argument can be repeated similarly near the other axes.
By translation invariance we can assume that $Q = \{|x_j| \leq
2^l\,:\, j=1,\dots, d \}$.
Then we consider a multiplier $\M$ of the form 
\[
i 2^j \M = m_l(x_1) \partial_1 + \partial_1 m_l(x_1) 
 \]
where $m_l(s) = m(2^{-l} s)$ with $m$ a smooth bounded increasing function 
with $m'(s) = \psi^2(s)$ for some Schwartz function $\psi$ 
localized at frequency $\lesssim 1$ with 
$\psi \sim 1$ for $|s| \leq 1$.

The properties (1) and (2) clearly hold for $\M$ and $u = u_j$ due to
the frequency localizations of $u_j$ and $m_l$.  It remains to
verify (3).  By the previous lemma applied for $g-I$, we can set
$A=-\Delta$. Then
\[
-i 2^j  [A,\M] = 2^{-l+2} \partial_1  \psi^2(2^{-l} x_1) \partial_1 + O(1).
\]
The last term is bounded, therefore 
\[
i 2^j \langle [A,\M] u_j, u_j \rangle =  2^{-l+2}  \|\psi(2^{-l} x_1) \partial_1 u_j\|_{L^2}^2 
+ O(\|u_j\|_{L^2}^2).
\]
Given the frequency and angular localization of $u_j$,  we obtain
\[
2^{-l} 2^{2j} \| \psi(2^{-l} x_1)  u_j\|^2_{L^2} \lesssim i 2^j \langle [A,\M] u_j, u_j \rangle
+ O(\|u_j\|_{L^2}^2).
\]
Hence (3) follows. Thus we have proved \eqref{le-l}
under the additional frequency localization condition \eqref{uj-ang}.

To prove \eqref{le-l} in general we use a wedge decomposition in the
frequency variables.  To this end, we consider a partition of unity $\{\theta_k(\omega)\}_{k=1}^{d}$,
\[
1 = \sum_k  \theta_k(\omega) \qquad \text{in } \ \mathbb{S}^{d-1},
\]
where, for each $k$,  $\theta_k(\omega)$ 
is supported in a small angle.   We then define the localized functions
$u_{j,k} = \Theta_{j,k}u_j$ via
\[
\mathcal{F}\Theta_{j,k}u = \theta_k\Bigl(\frac{\xi}{|\xi|}\Bigr)
\sum_{j-1\le l \le j+1} \phi_{l}(\xi)\hat{u}(t,\xi).
\]
These solve the equations
\[
(i\partial_t -A) u_{j,k} = \Theta_{j,k} f_{j}  - [A,\Theta_{j,k}] u_j.
\]
By Plancherel's theorem, it is trivial to see that $\Theta_{j,k}$ is
$L^2$ bounded.  We note further that the kernel of the operator
$\Theta_{j,k}$ has Schwartz class decay outside a ball of radius
$2^{-j}$.  Thus, it is easy to show that $\Theta_{j,k}$ is bounded on
$X$, and by duality on $Y$.

To prove \eqref{le-l} for $u_j$ we apply the appropriate multipliers to each of the $u_{j,k}$ and 
sum up. We obtain
\begin{equation}\label{le-l1}
\begin{split}
\!\!\!  2^{j-l}\| u_j\|_{L^2(Q)}^2 \lesssim & \ 
\|u_j\|_{L^\infty L^2}^2 +
  \|u_j\|_{X_j} (  \|f_{j}\|_{Y_j}   + \!\! \sum_k \| [A,\Theta_{j,k}] u_j \|_{Y_j})
\\   + & (2^{-j} + \| g - I\|_{\lX^s}) \|u_j\|_{X_j}^2 
\end{split}
\end{equation}
It remains to estimate the commutator, which is done via \eqref{fe:xxycom}.
Then  \eqref{le-l} follows.



 We now show how \eqref{l1le} follows from \eqref{le}.  We consider a
 partition of unity $\chi_Q$ corresponding to cubes $Q$ of scale $M
 2^j$.
 Allowing rapidly decreasing tails, we can assume that the functions
 $\chi_Q$ are localized at frequencies $\lesssim 1$.  We can also assume 
that $\chi_Q$ are smooth on the $M 2^j$ scale, in particular
 \[
|\nabla \chi_Q|\lesssim (2^{j} M)^{-1}, \qquad |\nabla^2 \chi_Q|\lesssim (2^{j} M)^{-2}.
\]  
The functions $\chi_Q u_j$ solve
\[
(i\partial_t - A)(\chi_Q u_j) = \chi_Q
f_j -[A,\chi_Q] u_j.
\]
We apply \eqref{le} to each of the functions $\chi_Q u_j$ and add them up.
This gives 
\begin{equation} \label{le-Q}
\begin{split}
\sum_{Q}  \| \chi_Q u_{j} \|_{X_j } \lesssim & \ 
\sum_Q \| \chi_Q u_{0j} \|_{L^2} +   \| \chi_Q f_{j} \|_{Y_j} + \| [A,\chi_Q] u_j\|_{L^1 L^2}.
\end{split}
\end{equation}
It remains to estimate the commutators. Using the bounds on the derivatives of $\chi_Q$
we obtain
\[
\sum_Q \| [A,\chi_Q] u_j\|_{L^1 L^2} \lesssim M^{-1}  \sum_Q \| \chi_Q u_j\|_{L^\infty L^2}.
\]
Hence if $M$ is large enough (independently of $j$) then the last term on the right 
in \eqref{le-Q} can be discarded, and we are left with 
\begin{equation} \label{le-Q1}
\begin{split}
\sum_{Q}  \| \chi_Q u_{j} \|_{X_j } \lesssim 
\sum_Q \| \chi_Q u_{0j} \|_{L^2} +   \| \chi_Q f_{j} \|_{Y_j}.
\end{split}
\end{equation}
The transition from cubes of size $M2^j$ to cubes of size $2^j$ is straightforward, 
and \eqref{l1le} follows.

\end{proof}


\section{Proof of 
Theorem \ref{thm:main1}}
\label{sec:proof}

We recall that the equation \eqref{eqn:quasiquad} turns into an
equation of the form \eqref{eqn:quasiquad1} by differentiation. Hence
it suffices to prove part (b) of the theorem. We recast the equation
\eqref{eqn:quasiquad1} in a paradifferential form, given by
\begin{eqnarray}
\label{eqn:param1}
\left\{ \begin{array}{l}
L_{j} u_j  = f_j, \\ \\
u_j (0) = (u_{0})_{j},
\end{array} \right.
\end{eqnarray}
where
\begin{eqnarray*}
L_{j} = (i \p_t + \p_k g^{kl}_{<j-4} \p_l)
\end{eqnarray*}
and 
\begin{equation}\label{fjdef}
f_j = S_j F(u,\nabla u) - S_j  \partial_{k} g^{kl}_{>j-4}  \partial_l  u
- [S_j, \partial_{k} g^{kl}_{<j-4}\partial_l] u.
\end{equation}

\subsection{ A formal bootstrap}
Using the bounds in Proposition~\ref{prop:nonlinear} one can estimate 
the $f_j$'s by the following

\begin{lemma} \label{fjest}
Let $s> \frac{d}{2}+2$,   and  $u \in \lX^s$ with frequency envelope $\{ a_j \}$. Then  
the functions $f_j$ in \eqref{fjdef} satisfy
\begin{equation}\label{fjbd}
\|f_j\|_{\lY^s}  \lesssim a_j \|u\|_{\lX^s}^2 c(\|u\|_{\lX^s}  ).
\end{equation}
\end{lemma}


\begin{proof}
The first term is estimated using \eqref{fe:moser} followed by \eqref{fe:xxy1}
with $\sigma = s-1$, taking advantage of the fact that $F$ is at least quadratic 
at zero.  The second term is estimated using \eqref{fe:moser} and \eqref{fe:xxy3} with $\sigma = s$.
For the third term we use \eqref{fe:xxycom}.
\end{proof}

As a corollary of the above lemma it follows that
\[
\sum_j \|f_j\|_{\lY^s}^2 \lesssim  \|u\|_{\lX^s}^4 c(\|u\|_{\lX^s}  ).
\]
For each of the equations in \eqref{eqn:param1} we can apply 
Proposition~\ref{prop:morawetz}.  Square summing we obtain
\[
\| u\|_{\lX^s}^2 \lesssim \|u_0\|_{l^1 H^s}^2 + \|u\|_{\lX^s}^4 c(\|u\|_{\lX^s}).
\]
From here a continuity argument {\em formally} leads to 
\[
\| u\|_{\lX^s} \lesssim \|u_0\|_{l^1 H^s}
\]
assuming that the initial data $u_0$ is small enough.

\subsection{ The linear problem}

Here we consider the linear equation
\begin{eqnarray}
\label{lin1}
\left\{ \begin{array}{l}
(i \partial_t + \partial_{k} g^{kl}  \partial_l) u + V \nabla u + W u = h, \\ \\
u (0) = u_{0},
\end{array} \right.
\end{eqnarray}
and we prove the following:

\begin{prop}\label{p:lin1}
a) Assume that the metric $g$ and the potentials $V$ and $W$ satisfy 
\[
\| g - I \|_{\lX^s} \ll 1, \quad \|V\|_{\lX^{s-1}} \ll 1 , \quad \|W\|_{\lX^{s-2}} \ll 1  \qquad  s > \frac{d}2+2.
\]
Then the equation \eqref{lin1} is well-posed for initial data
$u_0 \in l^1 H^\sigma$ with $0 \leq \sigma \leq s-1$. 
and we have the estimate
\begin{equation}\label{bd:lin1}
\| u\|_{\lX^\sigma} \lesssim \|u_0\|_{l^1 H^\sigma} + \|h\|_{\lY^\sigma}.
\end{equation}

b) Assume in addition that $W=0$. Then the equation \eqref{lin1} is well-posed
 for initial data $u_0 \in l^1 H^\sigma$ with $0 \leq \sigma \leq s$, 
and the estimate \eqref{bd:lin1} holds.

\end{prop}

\begin{proof}
We rewrite the equation as a family of equations for the dyadic 
parts of $u$,
\[
\left\{ \begin{array}{l}
(i \partial_t + \partial_{k} g^{kl}_{<j-4}  \partial_l) u_j  = g_j+h_j, \\ \\
u_j (0) = u_{0j},
\end{array} \right.
\]
where
\[
g_j =  -S_j  \partial_{k} g^{kl}_{>j-4}  \partial_l  u_j
- [S_j, \partial_{k} g^{kl}_{<j-4}\partial_l] u_j - S_j V\nabla u -
S_j W u.
\]
As in Lemma~\ref{fjest}, we apply Proposition~\ref{fe:nonlinear} for each of the terms 
in $g_j$ to obtain
\[
\sum_j \|g_j\|_{\lY^\sigma}^2  \lesssim \|u\|_{\lX^\sigma}^2 (\|
g-I\|_{\lX^s}^2+ \|V\|_{\lX^{s-1}}^2 + \|W\|_{\lX^{s-2}}^2).
\]
The estimate \eqref{bd:lin1} follows by applying
Proposition~\ref{prop:morawetz} to each of these equations and summing
in $j$. The more restrictive range of $\sigma$ in part (a) arises due to the 
similar range in \eqref{fe:xxy1}.
\end{proof}


\subsection{The iteration scheme: uniform bounds}

Here we seek to construct solutions to \eqref{eqn:quasiquad1}
iteratively, based on the scheme
\begin{equation}
\label{iterate}
\left\{ \begin{array}{l}
(i \partial_t + \p_j g^{jk} (u^{(n)})  \p_k) u^{(n+1)} = 
F(u^{(n)},\nabla u^{(n)}) , \\ \\
u^{(n+1)} (0,x) = u_0 (x)
\end{array} \right. 
\end{equation}
with the trivial initialization 
\[
u^{(0)} = 0.
\]

Applying at each step Proposition~\ref{p:lin1} 
and assuming that $u_0$ is small in $l^1H^s$ we inductively 
obtain the uniform bound 
\begin{equation}\label{iterate-uniform}
\|u^{(n)}\|_{\lX^s} \lesssim \|u_0\|_{l^1 H^s}.
\end{equation}
Our next goal is to consider the convergence of this scheme.

\subsection{The iteration scheme: weak convergence}

Here we prove that our iteration scheme converges in the weaker $l^1 H^{s-1}$
topology. For this we write an equation for the difference
 $v^{(n+1)} = u^{(n+1)} - u^{(n)}$:
\begin{equation}
\label{iterate-diff}
\left\{ \begin{array}{l}
(i \partial_t + \p_j g^{jk} (u^{(n)})  \p_k) v^{(n+1)} =  V_n \nabla  v^{(n)}+ W_n v^{(n)},
 \\ 
v^{(n+1)}(0,x) = 0,
\end{array} \right. 
\end{equation}
where 
\begin{eqnarray*}
V_n & = & V_n(u^{(n)},\nabla u^{(n)},u^{(n-1)},\nabla u^{(n-1)}), \\
W_n & = & h_1( u^{(n)},u^{(n-1)})  + h_2 ( u^{(n)},u^{(n-1)})  \nabla^2 u^{(n)}.
\end{eqnarray*}
 For $V_n$ and $W_n$ by the Moser estimate \eqref{moser} we have
\[
\| V_n \|_{\lX^{s-1}} \ll 1, \qquad \| W_n \|_{\lX^{s-2}} \ll 1.
\]
This allows us to estimate the right hand side of \eqref{iterate-diff} 
in $\lY^{s-1}$ via \eqref{xxy1} and \eqref{xxy2}.
To estimate $v^{(n+1)}$ we use Proposition~\ref{p:lin1}.
We obtain
\begin{equation}
\| v^{(n+1)}\|_{\lX^{s-1}} \ll \|v^{(n)}\|_{\lX^{s-1}} .
\end{equation}
This implies that our iteration scheme converges in $\lX^{s-1}$ to some function $u$.
Furthermore, by the uniform bound \eqref{iterate-uniform} it follows that
\begin{equation}\label{mainb}
\| u\|_{\lX^s} \lesssim \|u_0\|_{l^1 H^s}.
\end{equation}
Thus we have established the existence part of our main theorem.

\subsection{ Uniqueness via weak Lipschitz dependence}

Consider the difference $v = u^{(1)} - u^{(2)}$ of two solutions. This 
solves an equation of the form \eqref{lin1} where
\begin{eqnarray*}
V & = & V(u^{(1)},\nabla u^{(1)},u^{(2)},\nabla u^{(2)}), \\
W & = & h_1( u^{(1)},u^{(2)})  + h_2 ( u^{(1)},u^{(2)})  \nabla^2 u^{(1)}.
\end{eqnarray*}
Applying Proposition~\ref{p:lin1}(a) we see that this equation is well-posed 
in $l^1 H^{s-1}$, and obtain the estimate
\begin{equation}\label{weak-lip}
\|  u^{(1)} - u^{(2)}\|_{\lX^{s-1}} \lesssim \|  u^{(1)}(0) - u^{(2)}(0)\|_{l^1 H^{s-1}}.
\end{equation}

\subsection{Frequency envelope bounds}

Here we prove a stronger frequency envelope version of the estimate \eqref{mainb}.
 
\begin{prop}\label{envelopes}
Let $u \in \lX^s$ be a small data solution to \eqref{eqn:quasiquad1}, which satisfies
\eqref{mainb}. Let $\{a_j\}$ be an admissible frequency envelope for the initial
data $u_0$ in $l^1 H^s$. Then $\{a_j\}$ is also a frequency envelope for $u$ 
in $\lX^s$.
\end{prop}

\begin{proof}
Define an admissible envelope $\{b_j\}$  for $u$ in $\lX^s$
by  
\begin{equation} \label{bjdef}
b_j =  2^{-\delta j} + \| u\|_{\lX^s}^{-1}  \max_k 2^{-\delta |j-k|}\| S_k u\|_{\lX^s}  .
\end{equation}
We estimate $u_j = S_j u$ using Proposition~\ref{p:lin1}
applied to the equation \eqref{eqn:param1}. For the functions $f_j$ we 
use Lemma~\ref{fjest} to obtain
\begin{equation}\label{fjbda}
       \|f_j\|_{\lY^s}    \lesssim b_j \|u\|^2_{\lX^s}  c(\|u\|_{\lX^s}).
\end{equation}
By Proposition~\ref{p:lin1}
applied to the equation \eqref{eqn:param1} we obtain
\[
\| S_j u\|_{\lX^s} \lesssim a_j\|u_0\|_{l^1 H^s} + b_j \|u\|_{\lX^s}^2  c(\|u\|_{\lX^s}).
\]
This implies that
\[
b_j \lesssim a_j \|u_0\|_{l^1 H^s}\|u\|_{\lX^s}^{-1}  + b_j \|u\|_{\lX^s}  c(\|u\|_{\lX^s}).
\]
Since $\|u\|_{\lX^s} $ is small and $\|u_0\|_{l^1 H^s} \lesssim
\|u\|_{\lX^s}^{-1}$, this implies that $b_j \lesssim a_j$, concluding
the proof.
\end{proof}

\subsection{ Continuous dependence on the initial data}
\label{sec:contdep}
Here we show that the map $u_0 \to u$ is continuous from $l^1 H^s$
into $\lX^s$.  

Suppose that  $u_0^{(n)} \to u_0$ in $l^1 H^s_x$. Denote by $a^{(n)}_j$, respectively
$a_j$ the frequency envelopes associated to $u_0^{(n)}$, respectively
$u_0$, given by \eqref{freqEnv}. 
If $u_0^{(n)} \to u_0$ in $l^1 H^s_x$ then $a^{(n)}_j \to a_j$ in $l^2$.
Then for each $\epsilon > 0$ we can find  some  $N_\epsilon$ so that 
\[
\| a^{(n)}_{> N_\epsilon} \|_{l^2} \leq \epsilon \qquad \text{for all $n$}.
\]
By Proposition~\ref{envelopes} we conclude that
\begin{equation}\label{hf-diff}
\| u^{(n)}_{> N_\epsilon} \|_{\lX^s} \leq \epsilon \qquad \text{for all $n$}.
\end{equation}

To compare $u^{(n)}$ with $u$ we use \eqref{weak-lip} for low frequencies 
and \eqref{hf-diff} for the high frequencies,
\[
\begin{split}
\| u^{(n)} - u\|_{\lX^s} \lesssim & \| S_{< N_\epsilon} (u^{(n)} - u)\|_{\lX^s} 
+  \| S_{> N_\epsilon} u^{(n)} \|_{\lX^s} +  \| S_{> N_\epsilon} u \|_{\lX^s} 
\\
\lesssim &  2^{N_\epsilon} \| S_{< N_\epsilon} (u^{(n)} - u)\|_{\lX^{s-1}} 
+  2 \epsilon
 \\
\lesssim &  2^{N_\epsilon} \| S_{< N_\epsilon} (u^{(n)}_0 - u_0)\|_{l^1 H^{s-1}} 
+  2 \epsilon.
\end{split}
\]
Letting $n \to \infty$ we obtain
\[
\lim \sup_{n \to \infty} \| u^{(n)} - u\|_{\lX^s} \lesssim \epsilon.
\]
Letting $\epsilon \to 0$ we obtain 
\[
\lim_{n \to \infty} \| u^{(n)} - u\|_{\lX^s} = 0,
\]
which gives the desired result.

\subsection{ Higher regularity}

Here we prove that the solution $u$ satisfies the bound
\begin{equation}
\| u\|_{\lX^\sigma} \lesssim \|u_0\|_{l^1 H^\sigma}, \qquad \sigma \geq s,
\end{equation}
whenever the right hand side is finite.

The idea is to repeatedly differentiate the equation. The simplest way
to do this would be to say that $\nabla u$ solves the linearized
equation. But this is like the difference equation and is well-posed
only in $l^1 H^{s-1}$ not in $l^1 H^s$.  Instead we redo the
computation as follows.  The original equation is
\[
(i \partial_t + \p_j g^{jk} (u)  \p_k) u =  F(u,\nabla u) .
\]
Differentiating we obtain 
\[
\begin{split}
(i \partial_t + \p_j g^{jk} (u) \p_k) (\partial_l u) = & - (g^{jk})' (u)
(\p_j \partial_l u \p_k u + \partial_l u \p_j \p_k u) \\
&+ F_{z_l} (u,\nabla
u) \nabla \partial_l u + F_{z_0} (u,\nabla u) \partial_l u.
\end{split}
\]
We write this in an abbreviated form as
\[
(i \partial_t + \p_j g^{jk} (u) \p_k) v_1 = G(u,\nabla u) \nabla v_1 + 
F_1(u,\nabla u)
\]
for $v_1 = \nabla u$, where $G(z) = O(|z|)$ and $F_1(z) = O(|z|^2)$ near $0$.
We know that $u$ is small in $\lX^s$, therefore, by Proposition~\ref{prop:nonlinear}
we get
\[
\| G(u,\nabla u)\|_{\lX^{s-1}} \ll 1, \qquad \| F_1(u,\nabla u)\|_{\lY^s}
\lesssim \|u\|^2_{\lX^s}.
\]
Hence using Proposition~\ref{p:lin1}(b) we obtain
\[
\| v_1\|_{\lX^s} \lesssim \|v_1(0)\|_{l^1 H^s} + \|u\|_{\lX^s}^2,
\]
which shows that
\[
\| u\|_{\lX^{s+1}} \lesssim \|u(0)\|_{l^1 H^{s+1}} + \|u\|_{\lX^s}^2.
\]

Inductively, we write an equation for $v_n= \nabla^n u$,
\[
(i \partial_t + \p_j g_{jk} (u) \p_k) v_n = G(u,\nabla u) \nabla v_n + 
F_n(u,\cdots,\nabla^{n} u)
\]
with the same $G$ as above. This leads to
\[
\| v_n\|_{\lX^s} \lesssim \|v_n(0)\|_{l^1 H^s} + \|u\|_{\lX^{s+n-1}}^2,
\]
which shows that
\[
\| u\|_{\lX^{s+n}} \lesssim \|u(0)\|_{l^1 H^{s+n}} + \|u\|_{\lX^{s+n-1}}^2.
\]

\end{document}